\documentclass[twoside,twocolumn]{article}

\usepackage{tabulary,graphicx,times,caption,fancyhdr,amsfonts,amssymb,amsbsy,latexsym,amsmath}
\usepackage[utf8]{inputenc}
\usepackage{url,multirow,morefloats,floatflt,cancel,tfrupee,textcomp,colortbl,xcolor,pifont}
\usepackage[nointegrals]{wasysym}
\usepackage{amsthm}
\usepackage{float}
\usepackage{bm}
\usepackage{tikz}
\usepackage{mathdots}
\usepackage{yhmath}
\usepackage{siunitx}
\usepackage{array}
\usepackage{gensymb}
\usepackage{tabularx}
\usepackage{booktabs}
\usepackage{subfigure}

\usepackage{lipsum}
\usepackage{mathtools}
\usepackage{cuted}
\usepackage{placeins}

\makeatletter

\usepackage{ifxetex}
\ifxetex\else
  \usepackage{dblfloatfix}
\fi

\@ifundefined{subparagraph}{
\def\subparagraph{\@startsection{paragraph}{5}{2\parindent}{0ex plus 0.1ex minus 0.1ex}%
{0ex}{\normalfont\small\itshape}}%
}{}

\def\URL#1#2{\@ifundefined{href}{#2}{\href{#1}{#2}}}

\def\UrlOrds{\do\*\do\-\do\~\do\'\do\"\do\-}%
\g@addto@macro{\UrlBreaks}{\UrlOrds}

\makeatother

\usepackage[paperheight=11in,paperwidth=8.3in,margin=2.5cm,headsep=.7cm,top=2.5cm]{geometry}
\usepackage[T1]{fontenc}

\renewenvironment{abstract}
	{\trivlist\item[]\leftskip0pt\par\vskip4pt\noindent
  	\textbf{\abstractname}\mbox{\null}\\}
	{\par\noindent\endtrivlist}

\def\keywords#1{\par\medskip\par\noindent\textbf{Keywords}: #1\par}

 \date{} \emergencystretch 8pt

\makeatletter
\def\author#1{\gdef\@author{\hskip-\tabcolsep%
	\parbox{\textwidth}{\raggedright\bfseries#1\\[1pc]}}}
\def\address[#1]#2{\g@addto@macro\@author{\\\hskip-\tabcolsep\parbox{\textwidth}{\raggedright%
	\normalsize\normalfont\textsuperscript{#1}#2}}}
\let\addresslink\textsuperscript
\def\correspondence#1{\g@addto@macro\@author{\\\hskip-\tabcolsep\parbox{\textwidth}{\raggedright%
	\vspace*{10pt}\normalsize\normalfont~\\#1~\\[12pt]}}}
\def\email#1{\g@addto@macro\@author{\\\hskip-\tabcolsep\parbox{\textwidth}{\raggedright%
	\normalsize\normalfont Emails: #1}}}

\def\title#1{\gdef\@title{\vspace*{-30pt}%
	\raggedright\textbf{\@journaltitle}~\\%
  \raggedright\bfseries\ifx\@articleType\@empty\vspace*{20pt}\else%
  \vspace*{20pt}\@articleType\vspace*{20pt}\\\fi#1}}
\let\@journaltitle\@empty \def\journaltitle#1{\gdef\@journaltitle{{\normalfont\itshape#1}}}
\let\@articleType\@empty \def\articletype#1{\gdef\@articleType{{\normalfont\itshape#1}}}

\let\@runningHead\@empty \def\RunningHead#1{\gdef\@runningHead{{\normalfont #1}}}

\usepackage{fancyhdr}
\fancypagestyle{headings}{\fancyhf{}
  \fancyhead[R]{\itshape\@runningHead}
  \fancyfoot[C]{\thepage}}
\makeatother

\usepackage[%
	numbers,sort&compress%
  ]{natbib}

\theoremstyle{plain}% Theorem-like structures provided by amsthm.sty
\newtheorem{theorem}{Theorem}%[section]
\newtheorem{lemma}[theorem]{Lemma}
\newtheorem{corollary}[theorem]{Corollary}
\newtheorem{proposition}[theorem]{Proposition}
\newcommand{\RNum}[1]{\uppercase\expandafter{\romannumeral #1\relax}}
\theoremstyle{definition}
\newtheorem{definition}[theorem]{Definition}

\theoremstyle{remark}
\newtheorem{remark}{Remark}

\usepackage{float,xcolor}

\begin{document}

% Title of the document
\title{Existence and Stability of $\alpha-$ harmonic Maps}

% Author names
\author{%
		Seyed Mehdi Kazemi Torbaghan\addresslink{1},
  	 Keyvan Salehi\addresslink{2} and
 	Salman Babayi\addresslink{3}
  }
		
% Affiliation
\address[1]{University of Bojnord,  Faculty of Basic Sciences, Department of Mathematics, Bojnord, Iran;}
\address[2]{Central of Theoretical Chemistry And Physics (CTCP), Massey University, Auckland, Newzealand;}
\address[3]{Department of Mathematics, Faculty of Sciences, Urmia University, Urmia, Iran.}

% Corresponding author details
\correspondence{Correspondence should be addressed to 
  Seyed Mehdi Kazemi Torbaghan: m.kazemi@ub.ac.ir}

% Emails of authors
\email{Keyvan Salehi(K.salehi@massey.ac.nz),  Salman Babayi(s.babayi@urmia.ac.ir)}

% Running Head
\RunningHead{Existence and Stability of  $\alpha-$ Harmonic Maps}

\maketitle 

% Abstract
\begin{abstract} In this paper,  we  first study 
the $\alpha-$energy functional,  Euler-Lagrange operator and $\alpha$-stress energy tensor.   Second,  it is shown that  the critical points of $\alpha-$ energy functional  are explicitly related to harmonic maps through conformal deformation.   In addition,  an $\alpha-$harmonic map is constructed from any smooth map between Riemannian manifolds under certain assumptions.   Next,  we determine the  conditions under which the fibers of horizontally conformal $\alpha-$ harmonic maps are minimal submanifolds.   Then,  the stability of any $\alpha-$harmonic map from a Riemannian manifold  to a Riemannian manifold with non-positive Riemannian curvature is  demonstrated.  Finally,   the instability of $\alpha-$harmonic maps  from a compact manifold to a standard unit sphere is investigated. 
\end{abstract}
\section{Introduction}
\noindent
Eells and Sampson introduced the concept of harmonic maps between Riemannian manifolds in 1964, \cite{js}.
If $\psi:(M,g )\longrightarrow (N,h)$ be a smooth map between Riemannian manifolds, Then the energy functional of $\psi $ is denoted by $E(\psi)$, and defined as follows 
\begin{equation}
E(\psi)=\frac{1}{2} \int _{M} \mid d\psi \mid ^2 dV _{g},
\end{equation}
where 
$dV_{g}$
is the volume element of $(M,g)$. 
The  Euler-Lagrange equation corresponding  to $E_{\psi},$ is given by
\begin{equation}\label{123er}
\tau (\psi):= Tr_{g } \,\nabla d \psi =0,
\end{equation}
where $\nabla $ is the induced connection on the pull-back bundle $\psi ^{-1}TN$. 
The section $\tau(\psi)\in \Gamma (\psi ^{-1}TN)$ described in \eqref{123er},  is known as the \textit{tension field} of $\psi$.\\
The authors showed in \cite{js},  that any smooth map  from a compact Riemannian manifold  to a Riemannian manifold  with non-positive sectional curvature can be deformed  into  a harmonic map.   This is well-known as the \textit{fundamental existence theorem for harmonic maps}.  Following that,  several researchers conducted a study on these maps, \cite{a1, h1,H, con11,z1}.  Harmonic maps have been examined in numerous theories in mathematical physics,  such as liquid crystal,  ferromagnetic material,  superconductors,  etc., \cite{hhhh,p}.\\
%
%A map $\psi : (M, g)\longrightarrow (P, \rho)$ between Riemannian manifolds is called biharmonic if $ \psi $ is a critical point of the bioenergy functional $ E_2 $ defined by
% \begin{equation}
% E_{2}(\psi):= \frac{1}{2}\int _{M}\mid \tau (\phi)\mid ^2d\upsilon _{g},
%\end{equation}
%where $ \mid\tau (\phi)\mid $ denotes the  Hilbert-Schmidt norm of the tension field of $ \phi $, and $ d\upsilon _g $ is the volume element of ,$(Mg)$. 
% The first and second variation formulas of the bienergy functional are obtained by Jiang in
%the concept of biharmonic maps, as an extension of harmonic maps, was first studied by Jiang in 1986, \cite{j2}. Due to the key role of these maps in describing the model of fluid dynamics and elasticity, many scholars have done research on this topic, \cite{3,549,546,547}.\\

Sacks and Uhlenbeck introduced perturbed energy functional that satisfies the Palais-Smale condition in their pioneering paper \cite{542} in 1981 to prove the existence of harmonic maps from a closed surface, and thus obtained so-called $\alpha$-harmonic maps, as critical points of perturbed functional, to approximate harmonic maps.
For $\alpha >1$, the \textit{$\alpha-$ energy functional} of the map $\psi$ is denoted by $E_{\alpha}(\psi)$ and defined as follows:
\begin{equation}
E_{\alpha}(\psi):=\int _{M}(1+\mid d\psi \mid ^{2})^{\alpha}dV _{g},
\end{equation}
where
$\mid d\psi \mid^{2}$
 denotes the Hilbert-Schmidt norm of the differential map $ d\psi \in \Gamma (T^{*}M\times \psi^{-1}TN)$ 
 with respect to $g$ and $h$. Noting that, for $\alpha >1$,
$E_{\alpha}$ satisfies
Morse theory and Ljusternik- Schnirelman theory, \cite{s1}.  Furthermore, the critical points of $E_{\alpha}$ are said to be the  
\textit{$\alpha$- harmonic maps}. 
Sacks and Uhlenbeck developed an existence theory for harmonic mappings of orientable surfaces into Riemannian manifolds using the critical  maps  of $E_{\alpha}$. Furthermore, they demonstrated that the convergence of the critical points of $E_{\alpha}$ is sufficient to construct at least one harmonic map of the sphere into any Riemannian manifold, \cite{542}. Many studies have recently been conducted on these maps. For example, the authors of \cite{s2} investigated the convergence behaviour of a sequence of $\alpha-$harmonic mappings $u_{\alpha}$ with $E_{\alpha}(u_{\alpha})<C$ from a compact surface $(M,g)$ into a compact Riemannian manifold $(N,h)$ without boundary.  It is worth noticing that this sequence  converge weakly to a harmonic map.  Furthermore,  in \cite{s1},  it is studied   the energy identity and necklessness for a sequence  of Sacks-Uhlenbeck maps during blowing up.
\\

Following the concepts of \cite{ara,Dj,js,s2, s1, H,p, 542}, we study the stability, existence, and structure of $\alpha$-harmonic maps, as well as their practical applications, in this work. The existence of $\alpha-$harmonic maps in an arbitrary class of homotopy,  and  the conditions under which the fibers of $\alpha-$harmonic maps are minimal submanifolds,  are explored in particular. Furthermore,  the stability of harmonic mappings from a Riemannian manifold to a non-positive Riemannian curvature and unit standard sphere is explored.
Sections 3-7 present our key findings. \\
This manuscript is organized as follows. Section 2 investigates the concepts of $\alpha-$ energy functional and $\alpha-$ harmonic maps, and it provides an explicit relationship between $\alpha-$ harmonic maps and harmonic maps via conformal deformation. 
Section 3 looks at the existence of $\alpha-$harmonic mappings.  It is demonstrated that, under certain assumptions,  every smooth map between Riemannian manifolds can be deformed  into an $\alpha-$harmonic map. 
The $\alpha-$stress energy tensor and its practical applications are discussed in Section 4.   It is also showed that the vanishing of the divergence of the $\alpha-$stress energy tensor of $\psi$ is identical to the vanishing of the $\alpha-$hormonicity of $\psi$. The concept of horizontally conformal maps and their physical applications are revisited in Section 5.  The criteria that cause the fibers of horizontally conformal $\alpha-$ harmonic mappings to be minimal submanifolds are then examined. 
In Section 6, the Jacobi operator and Green's theorem are used to derive the second variation formula of the $\alpha-$energy functional. 
The stability of $\alpha-$harmonic maps is explored in the last part.  Any $\alpha-$harmonic map from a Riemannian manifold to a Riemannian manifold with non-positive curvature is proved to be stable in this sense. It is also explored the instability of $\alpha-$harmonic mappings from a compact manifold to a unit standard sphere.

\section{\large{$\alpha$- harmonic maps}} 
\noindent

In this section,  firstly we study the notion of $\alpha-$ harmonic maps.  Then,  the 
certain relation between 
$\alpha$- harmonic maps and harmonic maps through conformal deformation is given. 
\\
Let $\psi: (M, g)\longrightarrow (N, h)$ be a smooth map between Riemannian manifolds.
Throughout this paper, it is considered that $ (M,g) $ is a compact Riemannian manifold. Furthermore,  the Levi-Civita connections on $ M$ and $ N $ are denoted by $ \nabla^{M} $ and $ \nabla ^N $, respectively.
Moreover, the induced connection on  the pullback bundle $ \psi^{-1}TN $ is denoted 
 by 
 $\nabla $ and
 defined by 
$ \nabla _{Y}V= \nabla ^{N} _{d\psi (Y)}V,$ 
for any smooth vector field $Y\in \chi(M)$ and section $ V \in \Gamma(\psi^{-1}TN )$.\\
Let $\alpha$ be a positive constant such that $\alpha>1$. Then, 
The $\alpha-$energy functional of $\psi $ is defined as follows 
\begin{equation}\label{34rtz}
E_{\alpha}(\psi):=\int _{M}(1+\mid d\psi \mid ^{2})^{\alpha}dV _{g}.
\end{equation}
The critical points of $E_{\alpha}$ are called $\alpha-$harmonic maps. 
By Green's theorem,
the corresponding Euler-Lagrange equation of the $\alpha-$energy functional,     $E_{\alpha}$,  is given by
 \begin{align}\label{h1}
\tau _{\alpha}(\psi)&:=2\alpha(1+\mid d\psi\mid^{2})^{\alpha-1}\tau(\psi)\nonumber \\&+d\psi(grad(2\alpha(1+\mid d\psi\mid^{2})^{\alpha-1})\nonumber \\&=0.
\end{align}
The section $\tau _{\alpha}(\psi)\in \Gamma(\psi ^{-1}TN)$ is said to be the $\alpha -$ tension field of $\psi$. \\
By \eqref{123er} and \eqref{h1}, we get the following theorem. 
\begin{theorem}
Let $\psi: (M, g)\longrightarrow (N, h)$ be a smooth map between 
Riemannian manifolds. Then, $\psi$ is $\alpha-$harmonic if and only if it has a vanishing $\alpha-$tension field.

\end{theorem}
%Now, an example of $\alpha-$harmonic map is given. 
%\begin{example}
%Let $M=\{x\in \mathbb{R}^{m}-\{0\}, \mid x\mid ^{2}>\sqrt{m+1}\}$ and 
%\begin{equation}
%\alpha(x)=m+\dfrac{k}{2e^{\mid x\mid ^{2}}+ln \mid x\mid ^{2}-ln m}
%\end{equation}
%is the smooth positive function on $M$. By \eqref{h1}, it can be seen that 
%\begin{equation}
% \psi:M\longrightarrow \mathbb{S}^{m-1}, \qquad x\longmapsto \dfrac{x}{\mid x\mid ^{2}},
% \end{equation} 
% is an $\alpha(.)$ Sacks-Uhlenbeck harmonic map. 
%\end{example}
\begin{definition}
Let 
$\psi:(M,g)\longrightarrow (N,h)$ be a smooth map between Riemannian manifolds. It is called that 
$\varphi$ is non-degenerate if the induced tangent map 
$\psi_{\ast}=d\psi$
is non-degenerate, i.e.  
$ker d\psi=\emptyset. $
\end{definition}
Now an explicit relation between 
$\alpha-$ harmonic maps and harmonic maps through conformal deformation is given. \\

\begin{proposition}\label{j23}
Let 
$\psi:(M^{m},g)\longrightarrow (N^{n},h)$ 
be a non-degenerate smooth map with
$m>2.$
Then, $\psi$ is an $\alpha-$ harmonic map if and only if the map 
$\psi$ is harmonic with respect to the conformally related metric 
$\bar{g}$
defined by 
\begin{equation}
\bar{g}=\{(2\alpha)^{\dfrac{2}{m-2}}(1+\mid d\psi\mid^{2})^{\dfrac{2\alpha-2}{m-2}}\}g.
\end{equation}
\end{proposition}
\begin{proof}
Assume that
$\bar{g}$
be a Riemannian metric for some smooth positive function $\mu$ on $M$ conformally associated to 
$g$
by 
$\bar{g}=\mu^{2} g.$
Denote the tension fields of the smooth map $\psi$ with respect to 
$g$
and  
$\bar{g}$
by 
$\tau(\psi)$
and 
$\bar{\tau}(\psi)$, 
respectively.  By 
\eqref{123er}, it can be obtained that 
\begin{equation} \label{345rt}
\bar{\tau}(\psi)=\dfrac{1}{\mu ^{m}}\{\mu ^{m-2}\tau(\psi)+ d\psi (grad \mu ^{m-2})\}. 
\end{equation}
Setting 
\begin{equation}\label{r5}
\mu =\{(2\alpha)^{\dfrac{1}{m-2}}(1+\mid d\psi\mid^{2})^{\dfrac{\alpha-1}{m-2}}\}
\end{equation}

By \eqref{h1}, \eqref{345rt} and \eqref{r5}, we get 
\begin{align}\label{345zh}
&\mu ^{m}\bar{\tau}(\psi)\nonumber \\&=\mu ^{m-2}\tau(\psi)+ d\psi (grad \mu ^{m-2})\nonumber \\ &=2\alpha(1+\mid d\psi\mid^{2})^{\alpha-1}
+ d\psi (grad(2\alpha(1+\mid d\psi\mid^{2})^{\alpha-1}) )
\nonumber \\&=\tau_{\alpha}(\psi).
\end{align}
The proposition \ref{j23} follows from 
\eqref{345zh}. 
\end{proof}

\section{The Existence of $\alpha-$ harmonic map}
In this section, it is assumed that 
$(M,g)$
and 
$(N,h)$
are compact Riemannian manifolds and $\mathcal{H}$
is a homotopy class of a smooth given map 
$\psi:(M^{m},g)\longrightarrow (N^{n},h)$. 
The following theorem is due to Hong.
\begin{theorem}\cite{H} \label{jk}
Suppose that 
$\psi:(M, g)\longrightarrow (N ,h)$
is a harmonic map. Then, for any $\hat{\varepsilon}>0$, there exists a smooth metric $\hat{g}$ conformally equivalent to $g$ such that $\psi:(M,\hat{g})\longrightarrow (N,h)$ is harmonic and $\mid d\psi \mid _{\hat{g}}\leq \hat{\varepsilon}$, where $\mid d\psi \mid _{\hat{g}}$ is the Hilbert-Schmidt norm with respect to $\hat{g}$ and $h$.  
\end{theorem}
By Theorem \ref{jk}, the following theorem follows.
\begin{theorem}
Let
$\psi \in \mathcal{H}$ and $\psi:(M^{m},g)\longrightarrow (N^{n},h)$
be a harmonic map. Then, there is a smooth metric $\bar{g}$ on $M$ conformally equivalent to $g$ such that $\psi:(M,\bar{g})\longrightarrow (N,h)$ is $\alpha-$harmonic if $m>2\alpha$.
\end{theorem}
\begin{proof}
Assume that 
$\varepsilon$ be a positive constant and  
$h(t)=(1+2t)^{\alpha}$.  
Let  
$k(t)=\dfrac{1}{2}\{(\dfrac{t}{2\alpha})^{\dfrac{1}{\alpha-1}}-1\}$
be the inverse function of 
$h'(t)$ on $[0,\varepsilon)$. Then, we have 
\begin{align}
h'(k(t))&=t,\label{e1} \\
k'(t)h''(k(t))&=1,\label{e2}
\end{align}
on $[0,\varepsilon)$. Setting 
\begin{equation}
y(u)=\dfrac{k(u^{m-2})}{u^{2}}.
\end{equation}
The derivative of $y$ with respect to $u$
can be obtained as follows
\begin{align}
\dfrac{dy}{du}&=\dfrac{1}{u^{3}}\{\dfrac{(m-2)u^{m-2}}{4\alpha(\alpha-1)}(\dfrac{u^{m-2}}{2\alpha})^{\dfrac{2-\alpha}{\alpha-1}}\nonumber \\&-((\dfrac{u^{m-2}}{2\alpha})^{\dfrac{1}{\alpha-1}}-1)\}.
\end{align}
Due to the fact that $k$ is the inverse function of $h'$ and $m>2\alpha$, we get 
\begin{align}
&\dfrac{dy}{du}((h'(t))^{\dfrac{1}{m-2}})\nonumber \\ &=\dfrac{\{(2m-4\alpha)t+(m-2)\}}{2(\alpha-1)\{2\alpha(1+2t)\}^{\dfrac{3(\alpha-1)}{m-2}}} \neq 0,
\end{align}
for $t\in [0,\varepsilon)$.
Then,  
$\dfrac{dy}{du}\neq 0, $   
on $u$ between $(h'(0))^{\dfrac{1}{m-2}}$ and $(h'(\varepsilon))^{\dfrac{1}{m-2}}$. 
Thus, there exists a smooth function $\theta(y)$ on $[0,\varepsilon')$ for some positive constant $\varepsilon'$, such that 
\begin{equation}\label{e4}
y=\dfrac{k((\theta(y))^{m-2})}{(\theta(y))^{2}}.
\end{equation}
Moreover, by \eqref{e1}, we have 
\begin{equation}\label{e5}
h'(k((\theta(y))^{m-2}))=(\theta(y))^{m-2}.
\end{equation}
By using \eqref{e4} and \eqref{e5}, we get 
\begin{equation}\label{e6}
h'(y(\theta(y))^{2})=(\theta(y))^{m-2}.
\end{equation}
Since $h'(t)$ is a positive function on $[0,\varepsilon')$,  by \eqref{e6}, it can be seen that $\theta$ is a positive function on $[0,\varepsilon')$.\\
Now,  by Theorem \ref{jk},  for  
$\hat{\varepsilon}: =\sqrt{\varepsilon'}$, there exists a smooth metric 
$\hat{g}$ conformally equivalent to $g$ such that $\psi:(M,\hat{g})\longrightarrow (N,h)$ is harmonic and 
 $\dfrac{\mid d\psi \mid _{\hat{g}}^{2}}{2}<\varepsilon'$.
  Due to the harmonicity of $\psi:(M,\hat{g})\longrightarrow (N,h)$,   the tension field of $\psi$ associated to the metrics $\hat{g}$ and $h$ vanishes,  
$\hat{\tau}(\psi)=Tr_{\hat{g}}d\psi=0$. Assume that 
$\bar{g}=\mu^{-2}\hat{g}$ for a smooth positive function $\mu :M\longrightarrow \mathbb{R}$, by \eqref{345rt}, we get 
\begin{align} \label{e7}
0&=\hat{\tau}(\psi)\nonumber \\&=\dfrac{1}{\mu ^{m}}\{\mu ^{m-2}\bar{\tau}(\psi)+ d\psi (grad_{\bar{g}} \mu ^{m-2})\}. 
\end{align}
Due to the fact that 
$\dfrac{\mid d\psi \mid _{\hat{g}}^{2}}{2}<\varepsilon'$,  the positive function $\mu$ can be defined as follows
\begin{equation}\label{e8}
\mu:=\theta(\dfrac{\mid d\psi \mid _{\hat{g}}^{2}}{2}).
\end{equation}
By \eqref{e6} and \eqref{e8}, it yields that 
\begin{align}\label{e9}
\mu ^{m-2}&=(\theta(\dfrac{\mid d\psi \mid _{\hat{g}}^{2}}{2}))^{m-2}=h'(\dfrac{\mid d\psi \mid _{\hat{g}}^{2}}{2}(\theta(\dfrac{\mid d\psi \mid _{\hat{g}}^{2}}{2}))^{2})
\nonumber \\
&=h'(\dfrac{\mid d\psi \mid _{\hat{g}}^{2}}{2}\mu ^{2})=h'(\dfrac{\mid d\psi \mid _{\bar{g}}^{2}}{2})\nonumber \\&=2\alpha(1+\mid d\psi \mid _{\bar{g}})^{\alpha-1}.
\end{align}
By substituting \eqref{e9} in \eqref{e7} and using 
\eqref{h1},  we get 
\begin{align}\label{345zu}
\bar{\tau}_{\alpha}(\psi) =&2\alpha(x)(1+\mid d\psi\mid^{2}_{\bar{g}})^{\alpha(x)-1}\bar{\tau}(\psi)\nonumber \\&+d\psi(grad_{\bar{g}}(2\alpha(x)(1+\mid d\psi\mid^{2}_{\bar{g}})^{\alpha(x)-1})\nonumber \\&=0,
\end{align}
and this completes the proof.  
\end{proof}
\section{The $\alpha-$stress energy tensor}
\noindent 
Let 
$\psi :(M^{m},g)\longrightarrow (N^{n},h)$ be a smooth map between Riemannian manifolds. The stress-energy tensor of $\psi$ associated with the $\alpha-$energy functional (in the sequel we call the $\alpha$-stress energy tensor of $\psi$,  in short) is denoted by 
$S_{\alpha}(\psi)$ and defined as follows
\begin{align}\label{4590}
S_{\alpha}(\psi)&=(1+\mid d\psi \mid ^{2})^{\alpha}g\nonumber \\& -2\alpha(1+\mid d\psi \mid ^{2})^{\alpha-1}\psi^{\ast}h.
\end{align}
\begin{remark}
The stress-energy tensor sometimes called the stress–energy-momentum tensor or the energy-momentum tensor is a tensor quantity in physics that describes the density and flux of energy and momentum in spacetime, generalizing the stress tensor of Newtonian physics. It is an attribute of matter, radiation, and non-gravitational force fields. This density and flux of energy and momentum are the sources of the gravitational field in the Einstein field equations of general relativity, just as mass density is the source of such a field in Newtonian gravity, \cite{h1}.\\
In general relativity, the symmetric stress-energy tensor acts as the source of spacetime curvature, and is the current density associated with gauge transformations of gravity which are general curvilinear coordinate transformations, \cite{m1}.  
\end{remark}

\begin{proposition}\label{14}
Based on the above notations, we get 
\begin{equation}
(div S_{\alpha}(\psi))(Y)=-h(\tau _{\alpha}(\psi), d\psi(Y)), 
\end{equation}
for all $Y\in \chi(M)$. 
\end{proposition}
\begin{proof}
Choose a local orthonormal frame 
$\{e_{i}\}_{i=1}^{m}$
on 
$M$ 
with 
$\nabla _{e_{i}}e_{j}\mid _{p}=0$,
at a point 
$p\in M$. 
Let $Y$ be a smooth vector field on 
$M$.  According to the definition of divergence operator on a Riemannian manifold, at point $p$, we have 
\begin{align}\label{1289}
&(div S_{\alpha}(\psi))(Y)
\nonumber \\
&=
-h(d\psi(grad (2
\alpha(1+\mid d\psi\mid ^{2})^{\alpha-1})), d\psi(Y))\nonumber \\
&+
2
\alpha(1+\mid d\psi\mid ^{2})^{\alpha-1}\sum_{j=1}^{m}\{
h(d\psi(e_{j}), \nabla _{Y}d\psi(e_{j}))\}
\nonumber \\
&-2
\alpha(1+\mid d\psi\mid ^{2})^{\alpha-1}\sum_{i=1}^{m}\{ h(\nabla _{e_{i}}d\psi(e_{i}), d\psi(Y))
\nonumber \\
&+
h( d\psi(e_{i}),
\nabla _{e_{i}}d\psi(Y)
-d\psi(\nabla ^{M}_{e_{i}}Y))\}. 
\end{align}
At point $p$,  
the last term of the right-hand side of \ref{1289}, can be obtained as follows 
\begin{align}\label{1290}
&\nabla _{e_{i}}d\psi(Y)-d\psi(\nabla^{M} _{e_{i}}Y)\nonumber \\
&=\nabla _{Y}d\psi(e_{i})+d\psi[e_{i},Y]-d\psi(\nabla _{e_{i}}^{M}Y)\nonumber \\
&=\nabla _{Y}d\psi(e_{i})+d\psi(\nabla ^{M}_{e_{i}}Y-\nabla ^{M}_{Y}e_{i})-d\psi(\nabla _{e_{i}}^{M}Y)\nonumber \\
%&=\nabla _{Y}d\psi(e_{i})-g(Y,e_{j}d\psi(\nabla _{e_{j}}^{M}e_{i}))\nonumber \\
&=\nabla _{Y}d\psi(e_{i}),
\end{align}
where we use
\begin{align}
\nabla ^{M}_{Y}e_{i}&=\sum_{j=1}^{m}\nabla ^{M}_{g(e_{j}, Y)e_{j}}e_{i}&\nonumber \\&=\sum_{j=1}^{m}g(e_{j}, Y)\nabla ^{M}_{e_{j}}e_{i}=0,
\end{align}
at the point $p$, for the third equality. 
Furthermore
\begin{align}\label{1291}
\nabla _{e_{i}}d\psi(e_{i})&=\nabla _{e_{i}}d\psi(e_{i})-d\psi(\nabla _{e_{i}}^{M}e_{i})\nonumber \\&=\tau(\psi).
\end{align}
By substituting \eqref{1290} and \eqref{1291} in the last equation of \eqref{1289}, we get
\begin{align}
&(div S_{\alpha}(\psi))(Y)\nonumber \\&=-2
\alpha(1+\mid d\psi\mid ^{2})^{\alpha-1})h(\tau(\psi), d\psi(Y))\nonumber \\&-h(d\psi(grad (2
\alpha(1+\mid d\psi\mid ^{2})^{\alpha-1}), d\psi(Y))
\nonumber \\&
=-h(\tau _{\alpha}(\psi), d\psi(Y)),
\end{align}
and hence completes the proof. 
\end{proof}
\begin{remark}
Based on Proposition \ref{14}, it can be seen that if  
$\psi $ is a submersion and $div S_{\alpha}(\psi)=0,$ then $\psi $ is an $\alpha-$harmonic map. Conversely, if $\psi $ is an $\alpha-$harmonic map then $div S_{\alpha}(\psi)=0$. 
\end{remark}

\section{Horizontal conformal $\alpha-$harmonic maps }
\noindent 
Let
$\psi :(M^{m},g)\longrightarrow (N^{n},h)$ be a smooth map between Riemannian manifolds and let 
$d\psi_{p}\neq 0$
for each 
$p\in M$.
The vertical space at 
$p$
with respect to 
$\psi$
is denoted by 
$V_{p}$
and defined by 
$V_{p}=ker \, d\psi_{p}$.
The orthogonal complement of 
$V_{p}$
in 
$T_{p}M$
is denoted by 
$H_{p}$
and called the 
horizontal space at 
$p$.
For any $X\in T_{p}M$, the tangent vector 
$X$ can be decomposed as follows
\begin{equation}
X=X^{H}+X^{V},
\end{equation}
here $X^{V}\in V_{p}$ and $X^{H}\in H_{p}$. The map 
$\psi$ is called a horizontally conformal map if there exists a positive function $\mu \in C^{\infty}(M)$ such that 
\begin{equation}
h(d\psi (X), d\psi (Y))=\mu^{2}g(X,Y),
\end{equation}
for any $X,Y\in H_{p}$ and $p\in M$. The function 
$\mu$ is said to be the dilation of $\psi$. \\

\begin{remark}
In Relativistic Astrophysics,  
 conformal mappings are used as a mathematical mechanism to obtain exact solutions to the Einstein field equations in
general relativity. The behavior of the space-time geometry quantities is given under a conformal transformation and the Einstein
field equations are exhibited for a perfect fluid distribution matter configuration, \cite{h1}.
\end{remark}

\begin{theorem}\label{3.44}
Let
$\psi :(M^{m},g)\longrightarrow (N^{n},h)$ be a horizontally conformal $\alpha-$harmonic map with dilation 
$\mu$
where 
$m>n$ and $n>2\alpha$. Then, the fibers of 
$\psi $ are minimal submanifold if and only if 
$grad \mu^{2}$ is vertical.  
\end{theorem}
\begin{proof}
For 
$p\in M$, 
choose a local orthonormal frame field 
$\{e_{i}\}_{i=1}^{m}$
near 
$p$,  such that 
$e_{1}, \cdots , e_{n}$
are horizontal vectors and 
$e_{n+1}, \cdots , e_{m}$
are vertical tangent vectors. Due to the fact that 
$\psi$
is horizontally conformal with dilation $\mu$, 
we get  
$\mid d\psi \mid ^{2}=n\mu ^{2}$. By \eqref{4590},  it is obtained that 
\begin{align}
S_{\alpha}(\psi)&=(1+n\mu ^{2})^{\alpha}g\nonumber \\& -2\alpha(1+n\mu ^{2})^{\alpha-1}\psi^{\ast}h.
\end{align}
By proposition 
\ref{14} 
and considering the fact that the map 
$\psi$
is 
$\alpha-$
harmonic, we get 
\begin{align}\label{5.1}
0&=div (S_{\alpha}(\psi))(e_{j})\nonumber \\&=
\sum _{i=1}^{m}(\nabla _{e_{i}}S_{\alpha}(\psi))(e_{i}, e_{j})
\nonumber \\&=n\alpha (1+n\mu ^{2})^{\alpha -1}e_{j}(\mu^{2})\nonumber \\&-\sum _{i=1}^{n}e_{i}\{2\alpha\mu^{2}(1+n\mu^{2})^{\alpha -1}h(d\psi(e_{i}), d\psi(e_{j}))\}\nonumber \\&+2\alpha(1+n\mu^{2})^{\alpha-1}\sum _{i=1}^{m}\{h(d\psi (\nabla ^{M}_{e_{i}}e_{i}), d\psi(e_{j}))\nonumber \\&+h(d\psi (e_{i}), d\psi (\nabla ^{M}_{e_{i}}e_{j}))\}.
\end{align}
For $j\quad (1\leq j \leq n)$, we have 
\begin{align}\label{5.2}
0&=\sum _{i=1}^{n}e_{i}(g(e_{i},e_{j}))\nonumber \\
&=\sum _{i=1}^{n}\{g(\nabla ^{M}_{e_{i}}e_{i}, e_{j})+g(e_{i}, \nabla ^{M}_{e_{i}}e_{j})\}
\nonumber \\
&=\sum _{i=1}^{n}\{g((\nabla ^{M}_{e_{i}}e_{i})^{H}, e_{j})+g(e_{i}, (\nabla ^{M}_{e_{i}}e_{j})^{H})\}
\nonumber \\
&=\dfrac{1}{\mu ^{2}}\sum _{i=1}^{n}\{h(d\psi (\nabla ^{M}_{e_{i}}e_{i})^{H}, d\psi(e_{j}))\nonumber \\
&+h(d\psi(e_{i}), d\psi(\nabla ^{M}_{e_{i}}e_{j})^{H})\}
\nonumber \\
&=\dfrac{1}{\mu ^{2}}\sum _{i=1}^{n}\{h(d\psi (\nabla ^{M}_{e_{i}}e_{i}), d\psi(e_{j}))\nonumber \\
&+h(d\psi(e_{i}), d\psi(\nabla ^{M}_{e_{i}}e_{j}))\}
\end{align}
By making use of \eqref{5.1} and \eqref{5.2}, for 
$1\leq j \leq n$, we get 
\begin{align}\label{5.3}
0&=n\alpha (1+n\mu ^{2})^{\alpha -1}e_{j}(\mu^{2})-e_{j}(2\alpha\mu^{2}(1+n\mu^{2})^{\alpha -1})\nonumber \\&+2\alpha(1+n\mu^{2})^{\alpha-1}\sum _{i=n+1}^{m}\{h(d\psi (\nabla ^{M}_{e_{i}}e_{i}), d\psi(e_{j}))\nonumber \\&+h(d\psi (e_{i}), d\psi (\nabla ^{M}_{e_{i}}e_{j}))\}\nonumber \\&=n\alpha (1+n\mu ^{2})^{\alpha -1}e_{j}(\mu^{2})-2\alpha(1+n\mu^{2})^{\alpha -1})e_{j}(\mu^{2})\nonumber \\&-2n\mu^{2}\alpha (\alpha -1)(1+n \mu^{2})^{\alpha -2}e_{j}(\mu^{2})
\nonumber \\&+
2\alpha(1+n\mu^{2})^{\alpha-1}\sum _{i=n+1}^{m}h(d\psi (\nabla ^{M}_{e_{i}}e_{i})^{H}, d\psi(e_{j}))
\nonumber \\&=\alpha(1+n\mu^{2})^{\alpha-2}[(n-2)+n\mu^{2}(n-2\alpha)]e_{j}(\mu^{2})\nonumber \\&+2\alpha \mu^{2}(1+n\mu^{2})^{\alpha-1}\sum _{i=n+1}^{m}g(\nabla ^{M}_{e_{i}}e_{i}, e_{j}). 
\end{align}
By \eqref{5.3} and the definition of the mean curvature vector 
,$H$, of the fiber of 
$\psi$, given as follows 
\begin{equation}
H=\dfrac{1}{m-n}\sum _{j=1}^{n}\sum _{i=n+1}^{m}g(\nabla ^{M}_{e_{i}}e_{i}, e_{j})e_{j},
\end{equation}
it is obtained that 
\begin{align}\label{67}
&\alpha (1+n\mu^{2})^{\alpha-2}[(n-2)+n\mu^{2}(n-2\alpha)](grad \mu^{2})^{H}\nonumber \\&+2(m-n)\alpha (1+n\mu^{2})^{\alpha-1}H=0.
\end{align}
From \eqref{67}   and the assumptions of this theorem, 
%Due to the fact that 
%$m>n$ , 
%$n>2\alpha$
%and 
%$\alpha >1$
it can be seen that 
\begin{equation}\label{68}
\alpha (1+n\mu^{2})^{\alpha-2}[(n-2)+n\mu^{2}(n-2\alpha)]>0.
\end{equation}
From \eqref{67} and \eqref{68}, the theorem \ref{3.44} follows. 
\end{proof}
\begin{corollary}\label{8o}
Under the assumptions of the theorem \ref{3.44}, 
%Let
%$\psi :(M^{m},g)\longrightarrow (N^{n},h)$ be a smooth map between Riemannian manifolds be a horizontally conformal $\alpha-$harmonic map with dilation 
%$\mu$
%where 
%$m>n$ and $n>2\alpha$. 
the fibers of 
$\psi $ are minimal submanifold if and only if 
the horizontal distribution has mean curvature 
$\dfrac{grad \mu ^{2}}{2\mu ^{2}}$.
\end{corollary}
\begin{proof}
By \eqref{5.1}, for $j \quad (n+1\leq j \leq m)$, we get
\begin{align}\label{5.6}
0&= \alpha (1+n\mu^{2})^{\alpha-2}\{ne_{j}(\mu^{2})\nonumber \\&+2\sum _{i=1}^{n}h(d\psi(e_{i}), d\psi(\nabla ^{M}_{e_i}e_{j})) \}.
\end{align}
For \, $i \quad (1\leq i \leq n )$, it is obtained that 
\begin{align}\label{5.7}
&h(d\psi(e_{i}), d\psi(\nabla ^{M}_{e_{i}}e_{j}))\nonumber \\&=h(d\psi(e_{i}), d\psi(\nabla ^{M}_{e_{i}}e_{j})^{H})\nonumber \\&=\mu^{2}g(e_{i}, (\nabla ^{M}_{e_{i}}e_{j})^{H})\nonumber \\&=\mu^{2}g(e_{i}, \nabla ^{M}_{e_{i}}e_{j})\nonumber \\&=-\mu^{2}g( \nabla ^{M}_{e_{i}}e_{i},e_{j}).
\end{align}
By \eqref{5.6} and \eqref{5.7}, we get 
\begin{equation}
ne_{j}(\mu^{2})-2\mu^{2}\sum _{i=1}^{n}g( \nabla ^{M}_{e_{i}}e_{i},e_{j})=0.
\end{equation}
Thus, the mean curvature 
,$\bar{H}$, of the horizontally distribution is obtained as follows
\begin{align}\label{5.20}
\bar{H}&=\dfrac{1}{n}\sum _{j=n+1}^{m}\sum _{i=1}^{n}g(\nabla ^{M}_{e_{i}}e_{i}, e_{j})e_{j}\nonumber \\&=\dfrac{1}{2\mu^{2}}\sum _{j=n+1}^{m}e_{j}(\mu^{2})e_{j}\nonumber \\&=\dfrac{1}{2\mu^{2}}(grad \mu^{2})^{V}. 
\end{align}
From the theorem \ref{3.44} and equation \eqref{5.20}, Corollary \ref{8o} follows. 
\end{proof}
%\begin{proof}

%\end{proof }
\section{Second variational formula }
\noindent
In this section, the second variation formula of the $\alpha-$energy functional is calculated by using the Jacobi operator and Green's theorem. \\
\begin{remark}
In physics,  the Jacobi operator plays a key role to solve the nonlinear
boundary value of Troesch's equation.  This equation arises in the investigation of confinement of a plasma column by radiation pressure and also in the theory of gas porous electrodes, \cite{m1,jk1}.
\end{remark}

\begin{theorem}(The second variation formula)\label{12sd}
Let 
$ \psi :(M,g)\longrightarrow (N,h) $ be an $\alpha-$ harmonic map and 
$\{\psi_{t,s}:M\longrightarrow N \}_{-\epsilon <s, t<\epsilon}$
be a 2-parameter smooth variation of $ \psi $
such that
$ \psi_{0,0}=\psi. $
 Then
 \begin{align}
&\dfrac{\partial ^{2}}{\partial t \partial s}E_{\alpha }(\psi)\mid _{t=s=0}=- \int _{M}h(J_{\alpha}(\upsilon), \omega),
 \end{align}
 where 
$$\upsilon= \dfrac{\partial \psi_{t,s}}{\partial t}\mid _{s=t=0}, \qquad \omega= \dfrac{\partial \psi_{t,s}}{\partial s}\mid _{s=t=0}, $$
 and 
 $J_{\alpha}(\upsilon)\in \Gamma (\psi ^{-1}TN)$ is given by 
 \begin{align}\label{36}
&J_{\alpha}(\upsilon)\nonumber \\&=2\alpha(1+\mid d\psi \mid ^{2})^{\alpha-1}Tr_{g} R^{N}(\upsilon, d\psi)d\psi 
\nonumber \\&+4\alpha(\alpha -1) Tr_{g} \nabla \langle \nabla \upsilon, d\psi\rangle (1+\mid d\psi \mid ^{2})^{\alpha-2}d\psi
\nonumber \\&+2\alpha Tr_{g}\nabla (1+\mid d\psi \mid ^{2})^{\alpha-1}\nabla \upsilon,
 \end{align}
 here $R^{N}$ is the curvature tensor on $(N,h)$ and $\langle ,\rangle$ denotes the inner product on $T^{\ast}M\otimes \psi ^{-1}TN. $
\end{theorem}

\begin{proof}
Let $(-\epsilon,\epsilon)\times (-\epsilon,\epsilon)\times M$ be a product manifold that is equipped with the product metric, and let  the natural extension of 
$\dfrac{\partial}{\partial t} $ 
on 
$(-\epsilon,\epsilon)$, 
$\dfrac{\partial}{\partial s} $
 on 
$(-\epsilon,\epsilon)$
and 
$\dfrac{\partial}{\partial x} $
on 
$M$
to 
 the product manifold 
 $(-\epsilon,\epsilon)\times (-\epsilon,\epsilon)\times M$ 
are denoted by 
$\dfrac{\partial}{\partial t}, \dfrac{\partial}{\partial s}$ 
and 
$\dfrac{\partial}{\partial x}$ again,
respectively.  Assume that  $\Psi: (-\epsilon,\epsilon)\times (-\epsilon,\epsilon)\times M\longrightarrow N $ be a smooth map defined by 
$\Psi(t,s,x)=\psi_{t,s}(x).$
The same notation 
$\nabla$
shall be used for the induced connection on 
$\Psi^{-1}TN$. 
Consider an orthonormal frame 
$\{e_{i}\}$
with respect to 
$g$
on 
$M$. 
 Then,  by \eqref{34rtz}, we have 

\begin{align}\label{cv4}
&\dfrac{\partial ^{2}}{\partial t \partial s}E_{\alpha }(\psi_{t,s})\mid _{s=t=0}\nonumber \\&= \int _{M}\dfrac{\partial ^{2}}{\partial t \partial s}(1+\mid d\psi _{t,s}\mid ^{2})^{\alpha}\mid _{s=t=0}dV _{g}. 
\end{align}
By computing the second derivative, we get 
\begin{align}\label{cv5}
&\dfrac{\partial ^{2}}{\partial t \partial s}(1+\mid d\psi _{t,s}\mid ^{2})^{\alpha}
\nonumber \\&=
(\dfrac{\partial A }{\partial s}) h(\nabla _{\dfrac{\partial }{\partial t}} d\psi _{t,s}(e_{i}),d\psi _{t,s}(e_{i}))
\nonumber \\&+
A\{h(\nabla _{\dfrac{\partial }{\partial t}} d\psi _{t,s}(e_{i}),\nabla _{\dfrac{\partial }{\partial s}}d\psi _{t,s}(e_{i}))
\nonumber \\&+
h(\nabla _{\dfrac{\partial }{\partial s}} \nabla _{\dfrac{\partial }{\partial t}} d\psi _{t,s}(e_{i}),d\psi _{t,s}(e_{i}))\}, 
\end{align}
where $A=2\alpha(1+\mid d\psi _{t,s}\mid ^{2})^{\alpha-1}$.
Due to the fact that  
\begin{align}
\dfrac{\partial A}{\partial s}\mid _{s=t=0}=B h(\nabla _{e_{i}}\omega, d\psi(e_{i})),
\end{align}
here $B=4\alpha(\alpha -1) (1+\mid d\psi \mid ^{2})^{\alpha-2}$, 
 the first term of the right-hand side of 
 \eqref{cv5}, can be obtained as follows 
 \begin{align}\label{cv6}
 & \dfrac{\partial A}{\partial s}
 h(\nabla _{\dfrac{\partial }{\partial t}} d\psi _{t,s}(e_{i}),d\psi _{t,s}(e_{i}))\mid _{s=t=0}\nonumber \\ &=B \langle \nabla \upsilon, d\psi\rangle h(\nabla _{e_{i}}\omega ,d\psi(e_{i}))
 \nonumber \\ &
 =e_{i}(h(\omega,B \langle \nabla \upsilon, d\psi\rangle d\psi(e_{i})))
 \nonumber \\ &-
 h(\omega, \nabla _{e_{i}}\{B\langle \nabla \upsilon, d\psi\rangle d\psi(e_{i})\}).
 \end{align}
 Furthermore, by calculating the second term of the right-hand side of 
 \eqref{cv5}, 
 we get 
 \begin{align}\label{cv7}
&Ah(\nabla _{\dfrac{\partial }{\partial t}} d\psi _{t,s}(e_{i}),\nabla _{\dfrac{\partial }{\partial s}}d\psi _{t,s}(e_{i}))\mid _{s=t=0}
\nonumber \\&=
e_{i}(h(2\alpha(1+\mid d\psi\mid ^{2})^{\alpha-1}\nabla _{e_{i}}\upsilon, \omega))
\nonumber \\&-h(\nabla _{e_{i}}(2\alpha\nabla _{e_{i}}(1+\mid d\psi \mid ^{2})^{\alpha-1}\nabla _{e_{i}}\upsilon), \omega).
 \end{align}
 By the definition and properties of the curvature tensor of
 $(N,h)$ and the compatibility of $\nabla$ with the metric $h$, the last term of the right-hand side of \eqref{cv5}, can be obtained as follows
 \begin{align}\label{cv8}
  &Ah(\nabla _{\dfrac{\partial }{\partial s}} \nabla _{\dfrac{\partial }{\partial t}} d\psi _{t,s}(e_{i}),
  d\psi _{t,s}(e_{i}))\mid _{s=t=0}
  \nonumber \\&
 =-2\alpha(1+\mid d\psi \mid ^{2})^{\alpha-1}h(R^{N}(\upsilon, d\psi(e_{i}))d\psi(e_{i}),\omega)  \nonumber \\&
 +e_{i}(h(\nabla _{\dfrac{\partial }{\partial s}}d\psi _{t,s}(\dfrac{\partial }{\partial t}),Ad\psi _{t,s}(e_{i})))\mid _{s=t=0}
\nonumber \\&
- h(\nabla _{\dfrac{\partial }{\partial s}}d\psi _{t,s}(\dfrac{\partial }{\partial t}), \nabla _{e_{i}}(Ad\psi_{t,s}(e_{i})) )\mid _{s=t=0}
  \end{align} 
  From equations \eqref{cv4},\eqref{cv5},\eqref{cv6},\eqref{cv7} and \eqref{cv8} and considering the Green's theorem and 
$\alpha-$ harmonicity of $\psi$, theorem \ref{12sd} follows. 
\end{proof}
\section{Stability of $\alpha$-harmonic maps}
In this section, the stability of $\alpha-$harmonic maps is studied.  First, we showed that any $\alpha-$harmonic map from a Riemannian manifold to a Riemannian manifold with non-positive curvature is stable. Then, the stability of $\alpha-$harmonic maps from a compact manifold to a unit standard sphere is investigated. \\
\begin{remark}
The stability of harmonic maps plays a key role in mathematical physics and mechanics, \cite{con11}.  For instance,  The linearized Vlasov-Maxwell equations are used to investigate harmonic stability properties for planar wiggler free-electron laser(FEL). It is worth noting that the analysis is carried out in the Compton regime for a tenuous electron beam propagating in the $z$ direction through the constant amplitude planar wiggler magnetic field $B^{0}=-B_{\omega}cos k_{0}z\hat{e}_{x}$, \cite{D1}.
\end{remark}

\begin{definition}\label{35}
Under the assumptions of theorem \ref{12sd}, 
setting
\begin{equation}
I(\upsilon,\omega)=\dfrac{\partial ^{2}}{\partial t \partial s}E_{\alpha }(\psi_{t,s})\mid _{s=t=0}.
\end{equation}
Then,  
$\psi$ is said to be stable $\alpha-$harmonic if $I(\upsilon,\upsilon)\geq 0$
for any compactly supported vector field $\upsilon$ along $\psi$. 
\end{definition}
From Theorem \ref{12sd}, we obtain the following corollary. 
\begin{corollary}\label{45}
Let $N$ be a Riemannian manifold with non-positive   Riemannian curvature.  Then,  any $\alpha-$harmonic map $ \psi :(M,g)\longrightarrow (N,h) $ is stable. 
\end{corollary}
\begin{proof}
Setting 
\begin{equation}
\delta(X)=2\alpha (1+\mid d\psi \mid ^{2})^{\alpha-1}h(\nabla _{X}\upsilon, \omega), 
\end{equation} 
and 
\begin{align}\label{ui}
 \qquad \eta (X)&=4\alpha(\alpha-1) (1+\mid d\psi \mid ^{2})^{\alpha-2}\nonumber \\& \qquad\langle \nabla\upsilon , d\psi \rangle h(d\psi(X),\omega),
\end{align}
for any $X\in \chi(M)$. Then, we have
\begin{align}\label{37}
&-h(2\alpha Tr_{g}\nabla (1+\mid d\psi \mid ^{2})^{\alpha-1}\nabla \upsilon,\omega)
\nonumber \\&=-\sum _{i=1}^{m}h(2\alpha\nabla_{e_{i}}(1+\mid d\psi \mid ^{2})^{\alpha-1}\nabla_{e_{i}}\upsilon,\omega)
\nonumber \\&=-\sum _{i=1}^{m}\{e_{i}(h(2\alpha(1+\mid d\psi \mid ^{2})^{\alpha-1}\nabla_{e_{i}}\upsilon,\omega))
\nonumber \\&+h(2\alpha(1+\mid d\psi \mid ^{2})^{\alpha-1}\nabla_{e_{i}}\upsilon,\nabla_{e_{i}}\omega)\}
\nonumber \\&=-div\delta+2\alpha(1+\mid d\psi \mid ^{2})^{\alpha-1}\langle \nabla \upsilon, \nabla \omega \rangle,
\end{align}
On the other hand, by \eqref{ui}, we get 
\begin{align}\label{38}
&-h( Tr_{g} \nabla \langle \nabla \upsilon, d\psi\rangle B d\psi, \omega)
\nonumber \\&=\sum _{i=1}^{m}\{-h( \nabla_{e_{i}} \langle \nabla \upsilon, d\psi\rangle  B d\psi(e_{i}), \omega)\}
\nonumber \\&=
\sum _{i=1}^{m}\{-e_{i}(h( \langle \nabla \upsilon, d\psi\rangle B d\psi(e_{i}), \omega))
\nonumber \\&+
h( \langle \nabla \upsilon, d\psi\rangle B d\psi(e_{i}), \nabla_{e_{i}}\omega)\}
\nonumber \\&=-div \,\eta+B\langle \nabla \upsilon, d\psi\rangle \langle \nabla \omega, d\psi\rangle, 
\end{align}
where $B=4\alpha(\alpha -1) (1+\mid d\psi \mid ^{2})^{\alpha-2}$. 
By substituting \eqref{37} and \eqref{38} in \eqref{36} and using Divergence theorem and Definition \ref{35}, we get 
\begin{align}\label{40}
&I(\upsilon,\upsilon)
\nonumber \\&=\int_{M}\{4\alpha(\alpha -1)(1+\mid d\psi \mid ^{2})^{\alpha-2}\langle \nabla \upsilon, d\psi\rangle ^{2}
\nonumber \\&
-2\alpha(1+\mid d\psi \mid ^{2})^{\alpha-1}h(Tr_{g} R^{N}(\upsilon, d\psi)d\psi, \upsilon )
\nonumber \\&
+2\alpha(1+\mid d\psi \mid ^{2})^{\alpha-1}\mid \nabla \upsilon \mid ^{2}
\}dV_{g}.
\end{align}
By \eqref{40} and the assumptions, the corollary \ref{45} follows. 
\end{proof}
Now, the stability of $\alpha-$harmonic maps from a compact without boundary Riemannian manifold to a standard unit sphere is studied. \\
Consider a unit standard sphere 
$\mathbb{S}^{n}$
as a submanifold in 
$(n+1)-$dimensional Euclidean space $\mathbb{R}^{n+1}$. Denote the Levi-Civita connections on 
$\mathbb{S}^{n}$ and 
$\mathbb{R}^{n+1}$
by $\nabla ^{S}$ and $\nabla ^{R}$, respectively. At $p\in \mathbb{S}^{n}$, any vector 
$V$ in $\mathbb{R}^{n+1}$ can be decomposed as follows 
\begin{equation}\label{s0}
V=V^{\top}+V^{\bot},
\end{equation}
where $V^{\bot}=\langle V, p\rangle p$ is the  normal part to 
$\mathbb{S}^{n}$
and 
$V^{\top}$ is the tangent part to 
$\mathbb{S}^{n}$.  
The second fundamental form of 
$\mathbb{S}^{n}$ 
in 
 $\mathbb{R}^{n+1}$
 is denoted by 
 $B$
and defined as follows 
\begin{equation}\label{s1}
B(X,Y)=-\langle X,Y\rangle p,
\end{equation}
where $X, Y$ are tangent vectors of 
$\mathbb{S}^{n}$  
at $p$ and 
$\langle ,  \rangle$ is the Euclidean metric on $\mathbb{R}^{n+1}$.  
 Furthermore, the shape operator 
$A^{W}$ corresponding to a normal vector field 
$W$
on $\mathbb{S}^{n}$ is defined by 
\begin{equation}\label{s2}
A^{W}(X)=-(\nabla ^{R}_{X}W)^{\top},
\end{equation}
where $X$ is a tangent vector field on 
$\mathbb{S}^{n}$. 
Noting that, the second fundamental form and the shape operator of $\mathbb{S}^{n}$ are satisfied by the following equation
\begin{align}\label{s3}
\langle B(X,Y), W\rangle &=\langle A^{W}(X),    Y\rangle\nonumber \\&=-\langle X,Y\rangle \langle p,W\rangle,
\end{align}
for any tangent vectors $X$ and $Y$ of $\mathbb{S}^{n}$ at $p$. \\
From the sectional curvature of $\mathbb{S}^{n}$ and using \eqref{s1},  \eqref{s2} and \eqref{s3}, the following lemma is obtained.  
\begin{lemma}\label{asdf}
Let 
$ \psi :(M,g)\longrightarrow \mathbb{S}^{n} $ be a smooth map and 
let $\Lambda$ is a parallel vector field in $\mathbb{R}^{n+1}$. Then, at $p\in \mathbb{S}^{n}$, we have 
\begin{itemize}
\item[1.] $\nabla _{X} \Lambda ^{\top}=A^{\Lambda ^{\bot}}(d\psi(X)),$
\item[2.] $\langle \nabla _{X} \Lambda ^{\top}, d\psi(X) \rangle =- \mid d\psi(X)\mid ^{2} \langle p, \Lambda \rangle,$
\item[3.] $\mid \nabla _{X} \Lambda ^{\top}\mid ^{2}=\mid d\psi(X)\mid ^{2}\langle p, \Lambda \rangle^{2},$
\item[4.] $\langle R^{S}(\Lambda ^{\top}, d\psi(X))d\psi(X), \Lambda ^{\top} \rangle\\=\mid d\psi(X)\mid ^{2}\mid \Lambda ^{\top}\mid ^{2}-\langle d\psi(X), \Lambda \rangle ^{2}.$
\end{itemize} 
where $\nabla $ and $R^{S}$ denote the induced connection on $\psi^{-1}T\mathbb{S}^{n}$ and the curvature tensor of    \,
$\mathbb{S}^{n}$, respectively.  
\end{lemma}
\begin{proof}
By \eqref{s0} and \eqref{s2}, we have 
\begin{align}\label{s4}
\nabla _{X} \Lambda ^{\top}&=\nabla _{d\psi(X)}^{S} \Lambda ^{\top}=(\nabla _{d\psi(X)}^{R} \Lambda ^{\top})^{\top}\nonumber \\&=(\nabla _{d\psi(X)}^{R} \Lambda- \Lambda ^{\bot})^{\top}=-(\nabla _{d\psi(X)}^{R} \Lambda ^{\bot})^{\top}
\nonumber \\&=
A^{\Lambda ^{\bot}}(d\psi(X)).
\end{align}
Making use of \eqref{s3} and \eqref{s4}, we get 
\begin{align}
\langle \nabla _{X} \Lambda ^{\top}, d\psi(X) \rangle &=\langle A^{\Lambda ^{\bot}}(d\psi(X)), d\psi(X) \rangle \nonumber \\&=- \mid d\psi(X)\mid ^{2}
\langle p, \Lambda ^{\bot} \rangle
\nonumber \\&=
- \mid d\psi(X)\mid ^{2} \langle p, \Lambda \rangle.
\end{align}
 and 
\begin{align}
\mid \nabla _{X} \Lambda ^{\top}\mid ^{2}&=\langle \nabla _{X} \Lambda ^{\top},  \nabla _{X} \Lambda ^{\top} \rangle
\nonumber \\&=\langle A^{\Lambda ^{\bot}}(d\psi(X)),  A^{\Lambda ^{\bot}}(d\psi(X)) \rangle
\nonumber \\&=
-\langle A^{\Lambda ^{\bot}}(d\psi(X)),  d\psi(X) \rangle \langle p,  \Lambda ^{\bot} \rangle\nonumber \\&=\langle d\psi(X),  d\psi(X) \rangle \langle p,  \Lambda ^{\bot} \rangle ^{2}
\nonumber \\&=
\mid d\psi(X)\mid ^{2}\langle p, \Lambda \rangle^{2}.
 \end{align} 
Due to the fact that  the sectional curvature of $\mathbb{S}^{n}$ is constant,  we get 
 \begin{align}
&\langle R^{S}(\Lambda ^{\top}, d\psi(X))d\psi(X), \Lambda ^{\top} \rangle
\nonumber \\&=\mid d\psi(X)\mid ^{2}\mid \Lambda ^{\top}\mid ^{2} -\langle d\psi(X), \Lambda^{\top} \rangle ^{2}
\nonumber \\&
=\mid d\psi(X)\mid ^{2}\mid \Lambda ^{\top}\mid ^{2} -\langle d\psi(X), \Lambda \rangle ^{2},
 \end{align}
 and hence completes the proof. 
\end{proof}
\begin{theorem}\label{900}
Let $M$ be a compact without boundary manifold and let
$ \psi :(M,g)\longrightarrow \mathbb{S}^{n} $ be an $\alpha-$harmonic map such that 
$\mid d\psi\mid ^{2} <\dfrac{n-2}{2\alpha -n}$. Then, 
$\psi$ is unstable.  
\end{theorem}
\begin{proof}
Let 
$\{\Lambda_{k}\}_{k=1}^{n+1}$
be a parallel orthonormal frame field in 
$\mathbb{R}^{n+1}$
and 
$\{e_{i}\}_{i=1}^{m}$ be a local orthonormal frame field on 
$M$
and let 
$R^{\mathbb{S}}$
and 
$\nabla$
denote the curvature tensor of \,
$\mathbb{S}^{n}$
and the induced connection on 
$\psi^{-1}(T\mathbb{S}^{n})$, respectively.  
By \eqref{40}, we have
\begin{align}\label{5tzo}
&\sum_{k=1}^{n+1}I(\Lambda_{k}^{\top}, \Lambda_{k}^{\top})\nonumber \\
&=\int _{M} \{B\sum_{k=1}^{n+1}
 \sum_{i=1}^{m}\langle \nabla _{e_{i}} \Lambda ^{\top}_{k} , d\psi(e_{i}) \rangle ^{2}
 \nonumber \\  &+2\alpha (1+\mid d\psi \mid ^{2})^{\alpha-1}\sum_{k=1}^{n+1}
 \sum_{i=1}^{m}\{ \mid \nabla _{e_{i}} \Lambda ^{\top}_{k}\mid ^{2} \nonumber \\  &- \langle R^{S}(\Lambda ^{\top}_{k}, d\psi(e_{i}))d\psi(e_{i}), \Lambda ^{\top}_{k} \rangle\}\}dV_{g}, 
\end{align} 
where $B=4\alpha(\alpha-1)(1+\mid d\psi \mid ^{2})^{\alpha-2} $.
According to the Lemma \ref{asdf}'s second assertion,    
the first term of the right-hand side of \eqref{5tzo} can be obtained as follows
\begin{align}\label{g1}
&\sum_{k=1}^{n+1}
 \sum_{i=1}^{m}\langle \nabla _{e_{i}} \Lambda ^{\top}_{k} , d\psi(e_{i}) \rangle ^{2}
\nonumber \\
=& \sum_{k=1}^{n+1}
 \sum_{i=1}^{m}(-\mid d\psi(e_{i})\mid ^{2} \langle p, \Lambda_{k} \rangle)^{2}
 \nonumber \\
=& \sum_{k=1}^{n+1} \mid d\psi\mid ^{4}\langle p, \Lambda_{k} \rangle ^{2}
 \nonumber \\
=& \mid d\psi\mid ^{4}\mid p\mid ^{2}
 \nonumber \\
=& \mid d\psi\mid ^{4}.
\end{align}
By virtue of Lemma \ref{asdf}'s third claim,    
the second term of the right-hand side of \eqref{5tzo} can be calculated as follows
\begin{align}\label{g2}
&\sum_{k=1}^{n+1}
 \sum_{i=1}^{m}\mid \nabla _{e_{i}} \Lambda ^{\top}_{k}\mid ^{2}
 \nonumber \\
&
 =\sum_{k=1}^{n+1}
 \sum_{i=1}^{m} \mid d\psi(e_{i})\mid ^{2}\langle p, \Lambda_{k} \rangle^{2}
  \nonumber \\
&
 =\sum_{k=1}^{n+1}\mid d\psi\mid ^{2}\langle p, \Lambda_{k} \rangle^{2}
  \nonumber \\
&
 =\mid d\psi\mid ^{2}\mid p\mid ^{2}
   \nonumber \\
&
 =\mid d\psi\mid ^{2}
\end{align}
Since the sectional curvature of
$\mathbb{S}^{n+1}$ is constant and using the Lemma \ref{asdf}'s  fourth assertion,  the last term of the right-hand side of \eqref{5tzo} can be calculated as follows 
\begin{align}\label{g3}
&\sum_{k=1}^{n+1}
 \sum_{i=1}^{m} \langle R^{S}(\Lambda ^{\top}_{k}, d\psi(e_{i}))d\psi(e_{i}), \Lambda ^{\top}_{k} \rangle\nonumber \\
 &=\sum_{k=1}^{n+1}
 \sum_{i=1}^{m} 
\mid d\psi(e_{i})\mid ^{2}\mid \Lambda_{k} ^{\top}\mid ^{2}-\langle d\psi(e_{i}), \Lambda_{k} \rangle ^{2} 
\nonumber \\
% &=-\mid d\psi\mid ^{2}+\mid d\psi\mid ^{2}\sum_{k=1}^{n+1}\mid \Lambda_{k} ^{\top}\mid ^{2}
% \nonumber \\
 &=-\mid d\psi\mid ^{2}+\mid d\psi\mid ^{2}\sum_{k=1}^{n+1}\mid \Lambda_{k} - \Lambda_{k} ^{\bot}\mid ^{2}
 \nonumber \\
 &=-\mid d\psi\mid ^{2}+\mid d\psi\mid ^{2}\sum_{k=1}^{n+1}\mid \Lambda_{k} -\langle\Lambda_{k} ,p\rangle p \mid ^{2}
 \nonumber \\
 &=-\mid d\psi\mid ^{2}+\mid d\psi\mid ^{2}\sum_{k=1}^{n+1}(\mid \Lambda_{k} \mid^{2} -\langle\Lambda_{k} ,p\rangle ^{2})
  \nonumber \\
% &=-\mid d\psi\mid ^{2}+\mid d\psi\mid ^{2}\{(n+1)-\sum_{k=1}^{n+1}\langle\Lambda_{k} ,p\rangle ^{2}\}
%  \nonumber \\
 &=-\mid d\psi\mid ^{2}+\mid d\psi\mid ^{2}\{(n+1)-\mid p \mid ^{2}\}
   \nonumber \\
 &=(n-1)\mid d\psi\mid ^{2}
\end{align}
By substituting 
\eqref{g1}, \eqref{g2} and  \eqref{g3} in 
\eqref{5tzo}, we get  
\begin{align}\label{5tzl}
&\sum_{k=1}^{n+1}I(\Lambda_{k}^{\top}, \Lambda_{k}^{\top})\nonumber \\
&=\int _{M} \{4\alpha(\alpha-1)(1+\mid d\psi \mid ^{2})^{\alpha-2} \mid d\psi\mid ^{4}\nonumber \\   &+2\alpha (1+\mid d\psi \mid ^{2})^{\alpha-1}(2-n) \mid d\psi\mid ^{2}\}dV_{g}
\nonumber \\
&=\int _{M} C\{(2-n)+(2\alpha-n)\mid d\psi\mid ^{2}\}dV_{g}
 \end{align}
where $C=2\alpha (1+\mid d\psi \mid ^{2})^{\alpha-1}\mid d\psi\mid ^{2}$. By \eqref{5tzl} and the assumptions, it is obtained that 
 \begin{equation}
 \sum_{k=1}^{n+1}I(\Lambda_{k}^{\top}, \Lambda_{k}^{\top})<0
 \end{equation}
Then, $\psi $ is unstable and hence completes the proof. 
\end{proof}

\section*{Conflicts of Interest}
The authors declare that there are no conflicts of interest regarding the publication of this paper.

%\bibliographystyle{styleone}
%\bibliography{myref2}

\end{document}